\documentclass[11pt,epsf]{article}
\usepackage{graphicx}
\usepackage{amsthm}
\usepackage{amsfonts}
\usepackage{amssymb}
\title{On the location of the complex conjugate zeros
  of the partial theta function}
\author{Vladimir Petrov Kostov\\ 
Universit\'e C\^ote d'Azur, CNRS, LJAD, France\\   
e-mail: vladimir.kostov@unice.fr} 
\date{}
\newtheorem{tm}{Theorem}

\newtheorem{rem}[tm]{Remark}
\newtheorem{rems}[tm]{Remarks}
\newtheorem{lm}[tm]{Lemma}
\newtheorem{ex}[tm]{Example}

\newtheorem{prop}[tm]{Proposition}
\newtheorem{nota}[tm]{Notation}
\begin{document} 
\maketitle

\begin{abstract}
We prove that for any $q\in (0,1)$, all complex conjugate pairs of zeros 
of the partial theta function $\theta (q,x):=\sum _{j=0}^{\infty}q^{j(j+1)/2}x^j$ 
with non-negative real part belong to the half-annulus
$\{$Re$(x)\geq 0,~1<|x|<5\}$, where the outer radius
cannot be replaced by a number 
smaller than $e^{\pi /2}=4.810477382\ldots$, and
that for $q\in (0,0.2^{1/4}=0.6687403050\ldots ]$,
  $\theta (q,.)$ has no zeros with non-negative real part.
  The complex conjugate pairs of zeros with
  negative real part belong to the left
  open half-disk of radius $49.8$ centered at the origin.\\ 


\end{abstract}

\section{Introduction}

In this article we study the location of the complex conjugate pairs of zeros
of the {\em partial theta function} defined as the sum of the series

\begin{equation}\label{equtheta}
  \theta (q,x):=\sum_{j=0}^{\infty}q^{j(j+1)/2}x^j
  \end{equation}
in which $(q,x)\in (0,1)\times \mathbb{R}$. We consider $x$ as a variable and
$q$ as a parameter. The name of the function matches its resemblance
with the {\em Jacobi theta function} $\Theta (q,x):=\sum_{j=-\infty}^{\infty}q^{j^2}x^j$, because
$\theta (q^2,x/q)=\sum_{j=0}^{\infty}q^{j^2}x^j$;
the function is ``partial'', because summation to $\infty$
is performed from $0$ and
not from $-\infty$. It satisfies the functional equation

\begin{equation}\label{equqqx}
  \theta (q,x)=1+qx\theta (q,qx)~.
  \end{equation}

The interest in the study of the analytic properties of $\theta$
is justified by its
applications in different domains. These include Ramanujan type $q$-series
(see~\cite{Wa}), asymptotic analysis (see~\cite{BeKi}), the theory of (mock)
modular forms (see \cite{BrFoRh}) and statistical physics and combinatorics
(see~\cite{So}). In \cite{BFM} and \cite{CMW} it is explained how the function
$\theta$ can be applied to the research of problems about asymptotics
and modularity of partial and false theta functions and their
interaction with representation theory and conformal field theory.

The partial
theta function has found its place in Ramanujan's lost notebook,
see~\cite{AnBe} and~\cite{Wa}; see also \cite{EM}
about Appell-Lerch sums (mock theta functions). Andrews-Warnaar identities for
the partial theta function are explored in~\cite{WM}, \cite{Wei}
and~\cite{Sun}. The
relationship between $\theta$ and Artin-Tits monoids are the object
of study in~\cite{FGM}. 
For the Pad\'e approximants of $\theta$
see~\cite{LuSa}.

A recent interest in the properties of $\theta$ is motivated by its connection
to section-hyperbolic polynomials. These are real polynomials
with positive coefficients all whose roots are real
negative and all whose finite sections (i.e. truncations)
are also with all roots real negative. This research domain was initiated
by Hardy, Petrovitch and Hutchinson (see~\cite{Ha}, \cite{Pe} and~\cite{Hu})
and continued in \cite{Ost}, \cite{KaLoVi} and \cite{KoSh}.

The analytic properties of the partial theta function are interesting in their 
own. In \cite{Pr} one can find an explicit combinatorial interpretation of the
coefficients of the leading root of $\theta$ as a series in~$q$. In papers
\cite{KoFAA}--\cite{KoAnn24} the author has explored some analytic properties
of~$\theta$, including the situation when $q$ and $x$ are complex,
with $(q,x)\in \mathbb{D}_1\times \mathbb{C}$, and the case when they are real,
with $q\in (-1,0)$.

For any $q\in (0,1)$ fixed, the partial theta function has infinitely-many
negative and finitely-many complex conjugate pairs of zeros, see~\cite{KoBSM1}.
It turns out that for $q\in (0,1)$, all complex conjugate pairs belong to the
domain

$$\mathcal{E}_+~:=~\{ ~x~\in ~\mathbb{C}~:~{\rm Re}~x~\in ~(-5792.7,0)~,~~
|{\rm Im}~x|~<~132~\}~
\cup ~\{ ~|x|~<~18~\} ~;$$ 
for $q\in (-1,0)$, the same statement holds true about the domain

$$\mathcal{E}_-~:=~\{ ~x~\in~\mathbb{C}~:~|{\rm Re}~x|~<~364.2~,~~
|{\rm Im}~x|~<~132~\} ~,$$
see \cite{KoFAA19}. These results are not trivial at all given that as
$q\rightarrow \pm 1^{\mp}$ convergence of the series (\ref{equtheta})
becomes worse.

Thus the complex zeros of $\theta$ remain ``not far'' from the origin for
any $q\in (-1,0)\cup (0,1)$. One can ask also the question how close to the
origin the zeros of $\theta$ can be found. In this direction the following
results
can be cited:

\begin{itemize}
\item (\cite[Theorem~1]{KoMatStud}) {\em For any fixed $q\in (0,1)$,
  the partial theta
  function has no zeros in the domain (see Fig.~\ref{Katsnelson})}

  $$\mathcal{D}:=\{ x~\in ~\mathbb{C}~:~|x|~\leq ~3~,~~{\rm Re}~x~\leq ~0~,~~
  |{\rm Im}~x|~\leq~
  3/\sqrt{2}~=~2.121320344\ldots \} ~.$$
\item
  (\cite[Proposition~2]{KoMatStud}) {\em For any $q\in (0,1)$ fixed, the
  function
  $\theta (q,.)$ has no real zeros $\geq -5$.}
\item
  (\cite[Theorem~1]{KoAnn24}) {\em For each $q\in (-1,0)\cup (0,1)$ fixed,
  the function $\theta$
  has no zeros in the closed unit disk $\overline{\mathbb{D}_1}$.} (This result
  is not true in the situation when $q$ and $x$ are complex, see an example
  in~\cite{KoAnn24}.)
  \end{itemize}

The first result of the present paper reads:

\begin{tm}\label{tm1}
  For $q\in (0,1)$, the complex conjugate pairs with non-negative real part
  (if any) of $\theta (q,.)$ belong to the half-annulus
  $\mathcal{A}:=\{ {\rm Re}\, x\geq 0,~1<|x|<5\}$ (see Fig~\ref{Katsnelson}).
\end{tm}

The proof is given in Section~\ref{secprtm1}.
It uses part (2) of Theorem~\ref{tm2}; the latter is proved independently of
Theorem~\ref{tm1}. Part (1) of Theorem~\ref{tm2} 
gives an idea
how far from optimal the size of the half-annulus $\mathcal{A}$~is:

\begin{tm}\label{tm2}
  (1) There exists a sequence of values $q_{s}$
  of $q$ and pairs of
  purely imaginary zeros of
  $\theta (q_{s},.)$ whose modulus tends
  as $s\rightarrow \infty$ to a quantity
  which is~$\geq e^{\pi /2}=4.810477382\ldots$.

  (2) For $q\in (0,0.2^{1/4}=0.6687403050\ldots ]$,
    the function $\theta (q,.)$ has no zeros with non-negative
    real part.
\end{tm}

Theorem~\ref{tm2} is proved in Section~\ref{secprtm2}.
The number $e^{\pi /2}$ is not far from $5$, see Theorem~\ref{tm1},
while $5$ is much smaller than $18$, see the domain
$\mathcal{E}_+$ above.

Our next theorem suggests a location of the complex zeros of $\theta$ with
negative real part, a location which is much smaller than $\mathcal{E}_+$:

\begin{tm}\label{tm3}
  The complex conjugate pairs of zeros of $\theta (q,.)$ which are with
  negative real part belong to the left
  open half-disk of radius $49.8$ centered at the origin.
  \end{tm}

The theorem is proved in Section~\ref{secprtm3}.
The next section reminds some analytic properties of the partial
theta function.

\section{The zeros and the spectrum of the partial theta function
\protect\label{seczerosspectrum}}

\subsection{Zeros and spectrum of $\theta$ for $q\in (0,1)$}

We define {\em the spectrum} $\Gamma$ of $\theta$ as the set of values of
the parameter $q$ for which $\theta (q,.)$ has a multiple zero (this notion
was introduced by B.~Shapiro in \cite{KoSh}). When $q$ and
$x$ are real, the spectrum consists of real values of $q$. When
$q$ and $x$ are complex, the spectrum contains also complex values, see
\cite[Proposition~8]{KoAA}. For $q\in (-1,0)$, properties of the spectrum of
$\theta$ can be found in~\cite{KoPRSE2}. In this subsection we remind
facts about the zeros and the spectrum of $\theta$ for $q\in (0,1)$:

\begin{tm}{\rm (\cite[Theorem~1]{KoBSM1})} (1) The spectrum $\Gamma$ consists of
  countably-many values of $q$ denoted by $0<\tilde{q}_1<\tilde{q}_2<\cdots$,
  where $\lim_{j\rightarrow \infty}\tilde{q}_j=1^-$.

  (2) For $\tilde{q}_N\in \Gamma$, the function $\theta(\tilde{q}_N,.)$ 
has exactly one multiple real zero which is of multiplicity
$2$ and is the rightmost of its real zeros.

(3) For $q\in (\tilde{q}_N,\tilde{q}_{N+1})$ (we set $\tilde{q}_0:=0$), 
the function $\theta$ has exactly $N$ complex conjugate pairs of zeros
(counted with multiplicity).
\end{tm}

\begin{rems}\label{remsspectrum}
  {\rm (1) All coefficients of $\theta$ as a function in $x$ being positive for
    $q\in (0,1)$, $\theta$ has no positive zeros.
\vspace{1mm}

    (2) As it was mentioned in \cite{KoSh}, the former students
    A.~Broms and I.~Nilsson have calculated
    the ﬁrst 25 spectral numbers with an accuracy of 12 decimal positions.
    Their list with 6 decimals reads:}

    $$\begin{array}{lllll}
      0.309249~,&0.516959~,&0.630628~,&0.701265~,&0.749269~,\\ \\
      0.783984~,&0.810251~,&0.830816~,&0.847353~,&0.860942~,\\ \\
      0.872305~,&0.881949~,&0.890237~,&0.897435~,&0.903747~,\\ \\
      0.909325~,&0.914291~,&0.918741~,&0.922751~,&0.926384~,\\ \\
      0.929689~,&0.932711~,&0.935482~,&0.938035~,&0.940393~.\end{array}$$
  {\rm Hence for $q\in (0,\tilde{q}_1=0.309249\ldots ]$,
  there are no complex zeros and
  for $q\in (\tilde{q}_1,\tilde{q}_2=0.516959\ldots )$,
  there is exactly one complex conjugate pair.
\vspace{1mm}

(3) For $q\in (0,\tilde{q}_1)$, all zeros of $\theta$ are negative, so
they form a strictly decreasing sequence: $\cdots <\xi_2<\xi_1<0$. The zeros
of each of its derivatives w.r.t. the variable $x$ are also real negative.  
For $q\in (0,\tilde{q}_1)$, it is true that:
\begin{enumerate}
  \item 
At
even (resp. at odd) zeros the function $\theta (q,.)$ is decreasing
(resp. increasing). 
  \item 
$\theta >0$ for $x\in (\xi_{2k+1},\xi_{2k})$ and
    $\theta <0$ for $x\in (\xi_{2k+2},\xi_{2k+1})$
    (see \cite[Proposition~6]{KoBSM1}).
  \item
    For $k\in \mathbb{N}^*$, one has $\theta (q,-q^{-k})\in (0,q^k)$ (see
    \cite[Proposition~9]{KoBSM1}). Hence for $q\in (0,\tilde{q}_1)$,
    the following inequalities hold true:

    \begin{equation}\label{equstring}
      -q^{-2k}~<~\xi_{2k}~<~\xi_{2k-1}~<~-q^{-2k+1}~.\end{equation}
\end{enumerate}

(4) As $q$ increases and takes the value $\tilde{q}_k$, $k=1$, $2$, $\ldots$,
    the zeros $\xi_{2k-1}$ and $\xi_{2k}$ coalesce to form a double zero
    $y_{k}$ and then a complex conjugate pair. If one denotes by $\xi_k^{\ell}$
    the real zeros of
    $\partial ^{\ell}\theta /\partial x^{\ell}(q,.)$,
    $\ell =0$, $1$, $\ldots$, then
    for $q\in (0,\tilde{q}_N)$, the following zeros $\xi_k^{\ell}$ are
    well-defined continuous real-valued functions in $q$:
    $\cdots <\xi_k^{\ell}<\cdots <\xi_{2N-1}^{\ell}<0$
    (\cite[Corollary~2]{KoBSM1}). Items 1 and 2 of part (3) of these remarks
    and the inequalities (\ref{equstring}) remain
    valid for any $q\in (0,1)$ for the indices $j$ 
    for which the corresponding zeros $\xi_{j}$ are
    well-defined.
    \vspace{1mm}
    
    (5) The following asymptotic
    formulae can be found in~\cite{KoBSM2}:

    \begin{equation}\label{equasympt}
      \tilde{q}_k=1-\pi /2k+(\ln k)/8k^2+O(1/k^2)~~~\, {\rm and}~~~\,
      y_k=-e^{\pi}e^{-(\ln k)/4k+O(1/k)}.
      \end{equation}
    %For $q=\tilde{q}_k$, the double zero $y_k$ belongs to the interval
    %$(-\tilde{q}_k^{-2k},-\tilde{q}_k^{-2k+1})$, so it can be represented
    %in the form
    %$y_k=-\tilde{q}_k^{-2k+1-\mu_k}$, $\mu_k\in (0,1)$. Thus
%
 %   $$\begin{array}{l}
  %    (1-\pi /2k+(\ln k)/8k^2+O(1/k^2))^{-2k+1-\mu_k}=
   %   e^{\pi -(\ln k)/4k+O(1/k)}~,~~~\, {\rm i.~e.}\\ \\
    %  (-2k+1-\mu_k)\ln (1-\pi /2k+(\ln k)/8k^2+O(1/k^2))=
     % \pi -(\ln k)/4k+O(1/k)~.
    %\end{array}$$
    %The left-hand side equals
%
 %   $$(-2k+1-\mu_k)(-\pi /2k+(\ln k)/8k^2+O(1/k^2))=$$

    (6) As $q$ varies in $(0,1)$, the zeros of the partial theta function
    do not go to or arrive from infinity. Indeed, the complex zeros remain in
    the domain $\mathcal{E}_+$. For the real zeros it is true that
    $\lim_{k\rightarrow\infty}\xi_kq^k=-1$ (\cite[Theorem~4]{KoBSM1}). That is,
    they can be approximated by the terms of a geometric progression with
    ratio $1/q$. One can also use as argument the fact that for $j$
    sufficiently large, there is just one zero $\xi_j$ such that
    $|q|^{-j+1/2}<|\xi_j|<|q|^{-j-1/2}$, see~\cite{KoAA}. Similar
    remarks are valid in the cases $q\in (-1,0)$ and~$q\in \mathbb{D}_1$.}
\end{rems}

\subsection{Katsnelson's contour}

We define {\em Katsnelson's contour} as the union of two arcs in
$\mathbb{R}^2$

$$(x,y)=(e^t\cos t,\pm e^t\sin t)~,~~~\, t\in [0,\pi ]~,$$
see
Fig.~\ref{Katsnelson}. In \cite{Ka}
the following result is to be found:

\begin{tm}
  The family of functions $f_{\varepsilon}(z)=\sum_{n=0}^{\infty}e^{-\varepsilon n^2}z^n$
  with $\varepsilon >0$ 
  converges as $\varepsilon \rightarrow 0^+$ to the function $1/(1-z)$
  locally uniformly w.r.t. $z$ in the interior of Katsnelson's contour.
  Outside that contour one has
  $\overline{\lim_{\varepsilon \rightarrow 0^+}}|f_{\varepsilon}(z)|=\infty$.
\end{tm}

\begin{figure}[htbp]
%\includegraphics[scale=0.5]{parthetanegfirstfour.eps}
%\centerline{\hbox{\includegraphics[scale=0.7]{parthetanegfirstfour.eps}}}
\centerline{\hbox{\includegraphics[scale=0.7]{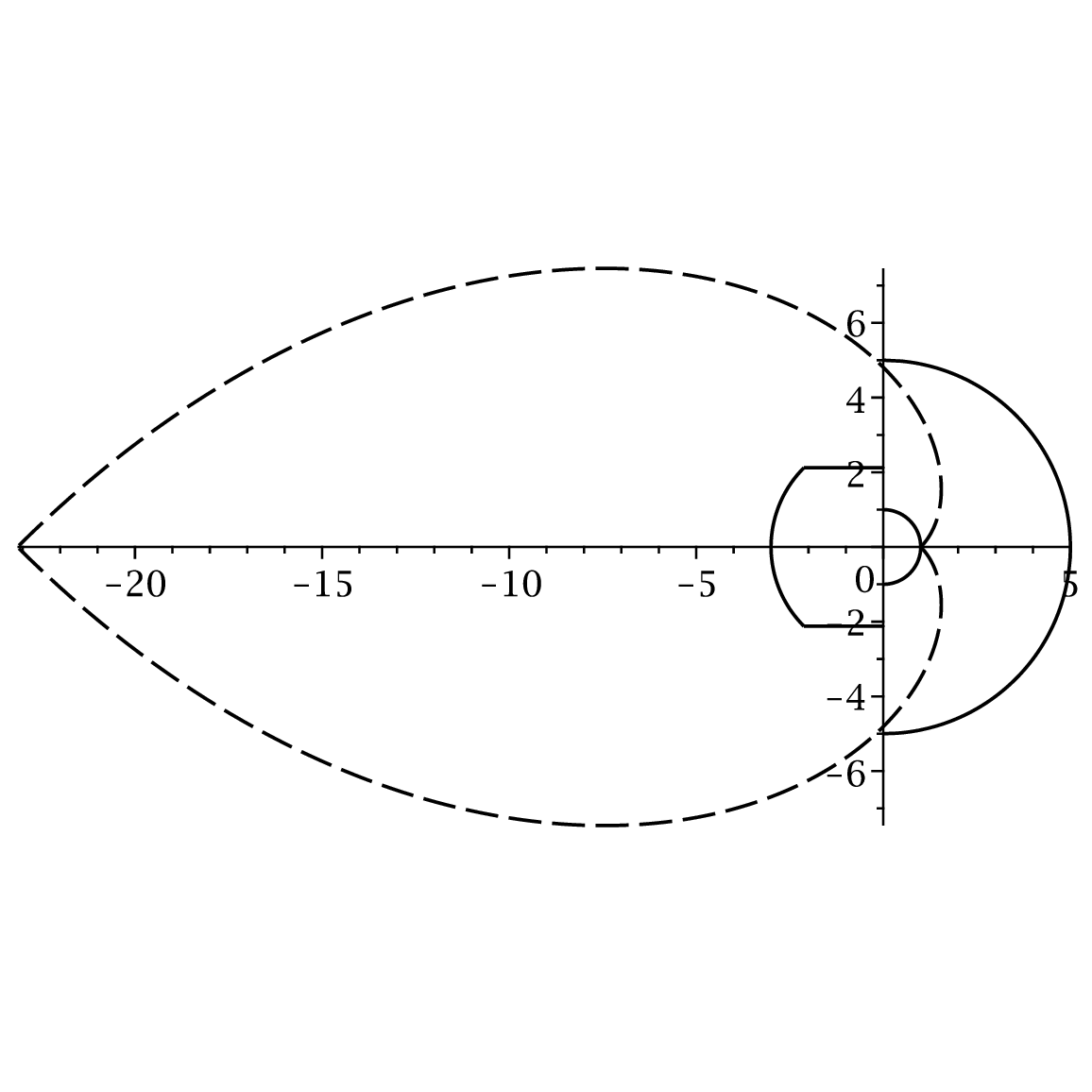}}}
%\centerline{\hbox{\epsfxsize=10cm \epsfbox{k=1234.pdf}}}
\caption{Katsnelson's contour (in dashed line) and the borders
  of the domain $\mathcal{D}$ and the half-annulus $\mathcal{A}$
  (in solid line).}
\label{Katsnelson}
\end{figure}

\begin{rems}\label{remsKats}
  {\rm (1) The domain inside Katsnelson's contour is much larger than
    the unit disk
in which the Taylor series of the function $1/(1-z)$ converges. Using the
notation of the present paper one can say that as $q\rightarrow 1^-$, 
$\theta (q,x)$ tends to $1/(1-x)$ locally uniformly inside
Katsnelson's contour. Indeed, set $q:=e^{-2\varepsilon}$. Then
$f_{\varepsilon}=\theta (q,z/\sqrt{q})$.
\vspace{1mm}

(2) The part of the domain inside Katsnelson's contour belonging to the right
half-plane is much smaller than its part in the left half-plane. This can be
compared to the fact that one finds much more precisely the possible location
of the complex zeros of $\theta$ with positive than of the ones with
negative real parts, see Theorems~\ref{tm1} and \ref{tm3}
and the domain $\mathcal{E}_+$.
\vspace{1mm}

(3) One can observe that to Katsnelson's contour belong the points
    $e^{(\pi /2)(1\pm i)}$, see~Theorem~\ref{tm2}.
    \vspace{1mm}

    (4) The angle between the tangent line to Katsnelson's contour and the
    radius-vector is everywhere equal to $\pm \pi /4$, so this line is
    vertical for $t=\pi /4$, i.~e. at the points

    $$e^{(\pi /4)(1\pm i)}~,~~~\, {\rm where}~~~\,
    e^{(\pi /4)}\cos (\pi /4)=e^{(\pi /4)}\sin (\pi /4)=1.550883197\ldots ~.$$
    It is horizontal at the points

    $$e^{(3\pi /4)(1\pm i)}~,~~~\, {\rm where}~~~\,
    e^{(3\pi /4)}\cos (3\pi /4)=-e^{(3\pi /4)}\sin (3\pi /4)=-7.460488536\ldots ~.$$
    
 (5) It would be interesting to know whether for any $q\in (\tilde{q}_1,1)$,
    all complex conjugate pairs of zeros of $\theta (q,.)$ belong to the
    interior of Katsnelson's contour.
    \vspace{1mm}

    (6) For $q=0.8$, the function $\theta_{100}(q,x):=\sum_{j=0}^{100}q^{j(j+1)/2}x^j$
    (the 100th truncation of $\theta$) has zeros

    $$b_{\pm}:=0.6128998489\ldots \pm 2.37247194\ldots i~~~{\rm of~modulus}~~~
    2.450361061\ldots$$
    lying inside Katsnelson's contour. Given that the
    101th term of $\theta$ is of modulus $<2\times 10^{-460}$ and that the next
    terms decrease in modulus faster than a geometric progression with ratio
    $3.2\times 10^{-10}$, and finally that the value of the derivative 

    $$(\partial \theta /\partial x)(0.8,b_{\pm})=
    -0.6143813197\ldots \mp 1.099995004\ldots i$$
    is of modulus $>1.25$, it seems likely that $\theta (0.8,.)$
    has conjugate zeros close to $b_{\pm}$. See Remark~\ref{remnum}
    for other similar numerical results.}
\end{rems}

    %Hence if one wants to give a more restrained
    %location of the zeros of $\theta$ by intersecting the half-annulus
    %$\mathcal{A}$ by half-planes of the form Re$(x)\leq a$, then one
    %should have $a\geq \tilde{a}:=e^{\pi /4}\cos (\pi /4)=1.550883197\ldots$.
    %Indeed,
    %by Theorem~\ref{tm2} there exists a sequence of complex conjugate pairs
    %of zeros of
    %$\theta (q_s,.)$ which tend to the points $\tilde{a}\pm i\tilde{a}$.
    %\vspace{1mm}

\section{Proof of Theorem~\protect\ref{tm1}\protect\label{secprtm1}}

\subsection{The method of proof\protect\label{subsecmethod}}

By the result of \cite{KoAnn24} cited before Theorem~\ref{tm1}
the partial theta function has no zeros with $|x|\leq 1$, so we prove only
that it has no zeros in the domain

$$\mathcal{F}~:=~\{ ~|x|~\geq ~5~,~~-\pi /2~\leq ~{\rm arg}\, x~\leq ~
\pi /2~\}~.$$
We consider the function $\Theta^*$
defined with the help of the
Jacobi theta function as follows:

$$\Theta^*(q,x)=\Theta (\sqrt{q},\sqrt{q}x)=
\sum_{j=-\infty}^{\infty}q^{j(j+1)/2}x^j~.$$
The function $\Theta$ can be represented by the Jacobi triple product 

$$\Theta (q,x^2)=\prod_{m=1}^{\infty}(1-q^{2m})(1+x^2q^{2m-1})(1+x^{-2}q^{2m-1})$$
from which one can deduce the equality

\begin{equation}\label{equT}
  \begin{array}{ccl}
    \Theta^*(q,x)&=&\prod_{m=1}^{\infty}(1-q^m)(1+xq^m)(1+q^{m-1}/x)\\ \\
    &=&(1+1/x)\prod_{m=1}^{\infty}(1-q^m)(1+xq^m)(1+q^m/x)~.\end{array}
\end{equation}
We set $G:=\sum_{j=-\infty}^{-1}q^{j(j+1)/2}x^j$. Thus $\theta =\Theta^*-G$. We prove
that in the domain $\mathcal{F}$ 
%$$\{ ~{\rm Re}~x~\geq ~0~,~~{\rm Im}~x~\geq ~0~\} ~\cap ~\{ ~|x|~\geq ~5~\}$$
one has $|\Theta^*|>|G|$ from which the theorem follows. For $|x|\geq 5$,
the function $G$ satisfies the inequality

\begin{equation}\label{equG}
  |G|\leq \sum_{j=1}^{\infty}5^{-j}=1/4~.
\end{equation}
We use the following lemma:

 \begin{lm}\label{lmarc}
      Suppose that $q\in [0.5,1)$ is fixed. Then:

      (1) For {\rm arg}$x\in [-\pi /2,\pi /2]$, if $|x|=:B\geq 5$ is fixed,
      then $|\Theta^*(q,x)|$
      is minimal for $x=\pm Bi$.      

      (2) For $x=\pm Ri$, $R\geq 5$, $|\Theta^*(q,x)|$
      is minimal for $x=\pm 5i$.
    \end{lm}

    \begin{proof}
      Part (1). Set $x:=Re^{i\varphi}$, $R>0$, $\varphi \in [-\pi /2,\pi /2]$. 
      Using the cosine rule one obtains

      $$\begin{array}{cccccc}
        |1+q^{m-1}/x|&=&|1+(q^{m-1}R^{-1})e^{-i\varphi}|&=&\sqrt{1+q^{2m-2}R^{-2}+
          2q^{m-1}R^{-1}\cos (\varphi )}&{\rm and}\\ \\
        |1+q^mx|&=&|1+(q^mR)e^{-i\varphi}|&=&\sqrt{1+q^{2m}R^{2}+
          2q^mR\cos (\varphi )}&\end{array}$$
      which quantities are minimal for $\varphi =\pm \pi /2$.
      \vspace{1mm}
      
      Part (2). For $m\geq 2$, the real and imaginary part of each product

      \begin{equation}\label{equ2factors}
        (1+q^mx)(1+q^m/x)=1+q^{2m}+iq^m(R-1/R)
        \end{equation}
      are non-decreasing in $R$ for $R\geq 5$ and the modulus of the product is 
      increasing. For the square of the modulus of the triple product

      $$T:=(1+1/x)(1+qx)(1+q/x)=1+q^2+q/R^2-q+i(-q/R+qR+1/R+q^2/R)$$
      one gets

      $$\partial (|T|^2)/\partial R=-4q^2/R^5-6q^2/R^3+2q^2R-2/R^3-2q^4/R^3$$
      which is positive for $R\geq 5$, $q\in [0.5,1)$. 
      \end{proof}

    For $(q,x)=(0.5,\pm 5i)$, one computes the values of $\Theta^*$ and
    $|\Theta^*|$ numerically:

    \begin{equation}\label{equnum}
      \begin{array}{ccl}
        \Theta^*(0.5,\pm 5i)&=&-1.542068340\ldots \pm 0.4429511372\ldots i~,\\ \\ 
      |\Theta^*(0.5,\pm 5i)|&=&1.604425279\ldots >1/4~,\end{array}
    \end{equation}
    see (\ref{equG}), so $\theta (0.5,.)$ has no zeros in the domain
    $\mathcal{F}$, see Lemma~\ref{lmarc}.

    \begin{prop}\label{propmain}
      For $q\in [0.5,1)$, one has $d(|\Theta^*(q,5i)|)/dq>0$
        (hence $d(|\Theta^*(q,-5i)|)/dq>0$).
    \end{prop}
    The proposition is proved in the next subsection. It implies that for
    $(q,x)\in [0.5,1)\times \mathcal{F}$,

      $$|\Theta^*(q,x)|-|G(q,x)|\geq |\Theta^*(0.5,5i)|-1/4>0~,$$
so $\theta$ has no zeros in $\mathcal{F}$ for $q\in [0.5,1)$. On the other
  hand, for $q\in (0,0.5]$, it has at most one complex conjugate pair
of zeros, see Remarks~\ref{remsspectrum}, 
%(see~\cite[Theorem~1]{KoBSM1}),
and when it does, the real part of the pair is negative, see part (2) of
Theorem~\ref{tm2}. This proves Theorem~\ref{tm1}.

      \subsection{Proof of Proposition~\protect\ref{propmain}}

      We prove that $d(|\Theta^*(q,5i)|^2)/dq>0$. Making use of equality 
      (\ref{equ2factors}) with $R=5$ one sees that

      \begin{equation}\label{equWm}
        W_m:=|(1+q^m5i)(1-q^mi/5)|^2=(1+q^{2m})^2+(5-1/5)^2q^{2m}=
        1+(626/25)q^{2m}+q^{4m}~.
        \end{equation}
      It is clear that
      (see~(\ref{equT}))

      \begin{equation}\label{equSW}
        \begin{array}{ccl}d(|\Theta^*(q,5i)|^2)/dq&=&|\Theta ^*|^2(S+W)~,~~~\,
        {\rm where}\\ \\
        S&:=&\sum_{m=1}^{\infty}(d(1-q^m)^2/dq)/(1-q^m)^2\\ \\ &=&
        -\sum_{m=1}^{\infty}2mq^{m-1}/(1-q^m)~,\\ \\
        W&:=&\sum_{m=1}^{\infty}(dW_m/dq)/W_m\\ \\ 
        &=&\sum_{m=1}^{\infty}((626/25)+2q^{2m})\cdot
        2mq^{2m-1}/(1+(626/25)q^{2m}+q^{4m})~,
        \end{array}
      \end{equation}
so $S<0$ and $W>0$. We first estimate the quantity $|S|$.

\begin{prop}\label{propS}
  One has $|S|\leq \pi^2/(3q\ln^2q)$.
\end{prop}

\begin{proof}
  For $x\geq 0$, we define the function $f$ as
  $f(x)=\left\{ \begin{array}{ccc}2xq^x/q(1-q^x)&{\rm for}&
x>0~,\\ \\ -2/(q\ln q)&{\rm for}&x=0~.\end{array}\right.$

\begin{lm}
The function $f$ is continuous and decreasing for $x\geq 0$, 
with $\lim_{x\rightarrow \infty}f(x)=0^+$.
\end{lm}

\begin{proof}
One finds that $\lim_{x\rightarrow 0^+}f(x)=f(0)$, so $f$ is continuous at~$0$.
One has

$$f'=2q^xg/q(1-q^x)^2~,~~~\, 
{\rm where}~~~\, g:=1+x\ln q -q^x~.$$
Clearly 
$g'=(\ln q)(1-q^x)$, with $\ln q<0$ and $1-q^x\geq 0$, so $g(0)=g'(0)=0$ and 
$g'\leq 0$, with equality only for $x=0$. Therefore for $x>0$, $g(x)<0$ 
and $f'<0$.

For the denominator one notices that for $x=0$, there is the equivalence 
$(1-q^x)^2\sim ((\ln q)x)^2$. As $g''=-(\ln^2q)q^x$, one obtains 
$f'(0)=-1/q$. Using the limits 

$$\lim_{x\rightarrow \infty}xq^x=0^+~~~\, {\rm and}~~~\,
\lim_{x\rightarrow \infty}(1-q^x)=1^-~,$$
one concludes that $\lim_{x\rightarrow \infty}f=0^+$. 
\end{proof}

The lemma implies that $|S|=f(1)+f(2)+\cdots \leq I:=\int_0^{\infty}f(x)dx$.
We set $a:=-\ln q$.
Integration by parts yields

$$I=[2x\ln (1-q^x)/aq]_0^{\infty}-(2/aq)\int_0^{\infty}\ln (1-q^x)dx~.$$
The first term to the right is interpreted as the difference of the limits at
the upper and lower bounds.
Both these limits equal~$0$. The second term is interpreted as

$$-(2/aq)\lim_{\alpha\rightarrow 0^+}\lim_{\omega\rightarrow\infty}
\int_{\alpha}^{\omega}\ln (1-q^x)dx~.$$
To find this term we set

$$\ln (1-q^x)=-\sum_{j=1}^{\infty}q^{jx}/j$$
(the series converges for $x>0$), thus

$$\int_0^{\infty}\ln (1-q^x)dx=\lim_{\alpha\rightarrow 0^+}\lim_{\omega\rightarrow\infty}
\sum_{j=1}^{\infty}[q^{jx}]_{\alpha}^{\omega}/(aj^2)=
-\sum_{j=1}^{\infty}1/(aj^2)=-\pi^2/(6a)$$
and $I=2\pi^2/(6qa^2)=\pi^2/(3qa^2)$. 
Therefore $|S|\leq \pi^2/(3qa^2)=\pi^2/(3q\ln^2q)$.
\end{proof}

\begin{prop}\label{propW}
  For $q\in [0.5,1)$, one has $W>4.46/(q\ln ^2q)$.
\end{prop}

\begin{proof}
  We remind that the quantities $W_m$ are defined in (\ref{equWm}). We set

  $$\begin{array}{l}
    V(y):=1+(626/25)y+y^2=(y+1/25)(y+25)~~~\, {\rm hence}~~~\,
    V'(y)=(626/25)+2y~~~\, {\rm and}\\ \\ 
    h(x):=((626/25)+2q^x)xq^{x-1}/(1+(626/25)q^x+q^{2x})=
    \psi (x)\phi (q^x)~,\end{array}$$
  where

  $$\left\{ \begin{array}{cclcc}
    \phi (y)&:=&V'(y)/V(y)&=&(\ln V(y))'~~~\, {\rm and}\\ \\
  \psi (x)&:=&xq^{x-1}~.&&\end{array}\right.$$
  As

$$\phi'(y)=-2(y^2+(626/25)y+(195313/625))/(V(y))^2~,$$
(the numerator has no real roots),
one has $\phi'(y)<0$ for $y>0$. The function $q^x$ (with $q\in (0,1)$)
being decreasing, the superposition $\phi (q^x)$ is continuous and
increasing for $x\geq 0$, with $\phi (q^0)=1$.

\begin{lm}\label{lmxi0}
  The derivative $h'$ has a single zero $\zeta_0/\ln (q)$,
  $\zeta_0=-2.685347089\ldots$.
\end{lm}

\begin{proof}
  Set $t:=(\ln q)x$, so $q^x=e^t$. With $V$ as above, one has
  $dh/dx=(\ln q)(dh/dt)$ and 

  $$\begin{array}{ccl}dh/dt&=&2e^tU(t)/((q\ln q)(V(e^t))^2)~,~~~\,
    {\rm where}\\ \\ 
    625U(t)&:=&625e^{3t}+7825te^{2t}+23475e^{2t}+1250te^t+196563e^t+7825t+
    7825~.\end{array}$$
  The derivative $dU/dt$, where 

  $$625dU/dt=1875e^{3t}+15650te^{2t}+54775e^{2t}+1250te^t+197813e^t+7825~,$$
  is easily shown
  to take only positive values (the minimal value of $te^t$ and $2te^{2t}$ is
  $-e^{-1}>-0.37$). As

  $$\lim_{t\rightarrow -\infty}U(t)=-\infty ~~~\, {\rm and}~~~\,
  \lim_{t\rightarrow \infty}U(t)=\infty ~,$$
  one concludes that $U$ has a unique zero $\zeta_0$.
  Numerical computation yields
  $\zeta_0=-2.685347089\ldots$.
\end{proof}

We remind that the terms of the sum $W$ (see (\ref{equSW}))
equal $h_m:=h(x)|_{x=2m}$. We define the index $m_0\in \mathbb{N}$ by the
condition

$$2m_0\leq \zeta_0/(\ln q)<2(m_0+1)~.$$
Lemma~\ref{lmxi0} implies that the sum $W$
is minorized by the
difference

$$I_*-2h(\zeta_0/(\ln q))~,~~~\, I_*:=\int_0^{\infty}h(x)dx~.$$ 
Indeed, 
for $m\leq m_0$ (resp. for $m\geq m_0+1$), 
the quantity $h_m$ can be represented as (the surface of) a rectangle 
with basis the interval $[2(m-1),2m]$ (resp. $[2m,2(m+1)]$), of length $2$,
and height $h_m$. 
For $m\leq m_0$
(resp. for $m\geq m_0+1$),
the union of these rectangles forms a figure containing the surface between
the graph of the function $h(x)$ and the abscissa-axis for
$x\in [0,2m_0]$ (resp. for $x\in [2(m_0+1),\infty)$). Not covered by these
  rectangles is the figure representing
  $\int_{2m_0}^{2(m_0+1)}h(x)dx<2h(\zeta_0/(\ln q))$.

  The change of variables $t=|\ln q|x$ transforms the integral $I_*$ into
$(1/(q\ln^2q))I_{\dagger}$, with 

$$\begin{array}{ccl}
  I_{\dagger}&:=&
  \int_0^{\infty}t\left( ((626/25)e^{-t}+2e^{-2t})/(1+(626/25)e^{-t}+e^{-2t})\right)
  dt=
  \kappa_{\dagger}~,\\ \\ 
  \kappa_{\dagger}&:=&6.82551484\ldots ~.\end{array}$$
  On the other hand

$$\begin{array}{ccl}2h(\zeta_0/(\ln q))&=&
2(1/(q|\ln q|))\left( t((626/25)e^{-t}+2e^{-2t})/(1+(626/25)e^{-t}+e^{-2t})
\left| _{t=|\zeta_0|}\right. \right) \\ \\ 
&=&2r_0\cdot (|\ln q|/(q\ln ^2q))~,~~~\, r_0:=1.699895161\ldots ~,\end{array}$$
  so for $q\in [0.5,1)$, one has

    $$\begin{array}{cc}|\ln q|\leq |\ln 0.5|=0.6931471806\ldots ~,&
      2h(\zeta_0/(\ln q))\leq 2r_0\cdot (|\ln 0.5|/(q\ln ^2q))~,\\ \\
      2r_0\cdot |\ln 0.5|<2.358&{\rm and}\end{array}$$

  $$W\geq I_*-2h(\zeta_0/(\ln q))>
  (\kappa_{\dagger}-2r_0\cdot |\ln 0.5|)/(q\ln ^2q)>
    4.46/(q\ln ^2q)~.$$
    \end{proof}
Recall that $W>0$ and $S<0$, see (\ref{equSW}). Propositions~\ref{propS}
and \ref{propW} imply that for $q\in [0.5,1)$,
  the following inequalities hold true:

  $$W>4.46/(q\ln ^2q)>\pi^2/(3q\ln^2q)\geq |S|~~
  ({\rm with}~\pi^2/3=3.28\ldots ).$$
Hence $d(|\Theta^*(q,5i)|^2>0$.

\section{Proof of Theorem~\protect\ref{tm2}\protect\label{secprtm2}}

\subsection{Proof of part (1)}

We use the following representation of $\theta$:

\begin{equation}\label{equevenodd}
  \theta (q,x)=\theta (q^4,x^2/q)+qx\theta (q^4,qx^2)~.
  \end{equation}
We set $v:=q^4$, so

\begin{equation}\label{equv}
  \theta (q,x)=\theta (v,v^{-1/4}x^2)+v^{1/4}x\theta (v,v^{1/4}x^2)~.
  \end{equation}
The function $\theta$ has a zero on the imaginary axis if and only if for
$x\in i\mathbb{R}$, the
functions $\theta (v,v^{-1/4}x^2)$ and $\theta (v,v^{1/4}x^2)$
have a common real zero. Indeed, if one sets $x:=iy$, $y\in \mathbb{R}$, then

\begin{equation}\label{equReIm}
  \begin{array}{l}
    {\rm Re}(\theta (q,x))=\psi_1(v,y):=\theta (v,-v^{-1/4}y^2)~~~\,
    {\rm and}\\ \\ 
    {\rm Im}(\theta (q,x))=\psi_2(v,y):=v^{1/4}y\theta (v,-v^{1/4}y^2)~.
  \end{array}
  \end{equation}
The positive zeros of $\psi_1(v,.)$ and $\psi_2(v,.)$ equal

\begin{equation}\label{equchimuxi}
  \chi_k:=v^{1/8}(-\xi_k^*)^{1/2}~~~\, {\rm and}~~~\,
\mu_k:=v^{-1/8}(-\xi_k^*)^{1/2}~~~\, {\rm respectively,~so}~~~\,
\chi_k=v^{1/4}\mu_k~,
\end{equation}
where $\xi_k^*$ stand for the zeros of $\theta (v,.)$.
We use the following result (see \cite[Theorem~4]{KoBSM1}): one has
$\lim_{k\rightarrow \infty}\xi_kq^k=-1$. Hence

\begin{equation}\label{equxik}
  \lim_{k\rightarrow\infty}\xi_k^*v^k=-1~,~~~\,  
\lim_{k\rightarrow\infty}\chi_k^2v^{k-1/4}=1~~~\, {\rm and}~~~\,
\lim_{k\rightarrow\infty}\mu_k^2v^{k+1/4}=1~.
\end{equation}
Thus for $k$ sufficiently large,

\begin{equation}\label{equchi}
  \chi_{k-1}<\mu_{k-1}<\chi_k<\mu_k~.
  \end{equation}
Suppose that for $v=v_0\in (0,1)$ and for $k\geq 2k_1$, the inequalities
(\ref{equchi}) hold true. Fix $k_2\geq k_1$ and start increasing $v$. When $v$
takes the value $\tilde{q}_{k_2}$, 
one has

$$\chi _{2k_2-1}=\chi_{2k_2}<
\mu_{2k_2-1}=\mu_{2k_2}~.$$
Hence
for some $v=v_{k_2}^*\in (v_0,\tilde{q}_{k_2})$, it is true that

$$\chi_{2k_2-1}<\mu_{2k_2-1}=\chi_{2k_2}<\mu_{2k_2}~.$$
i.~e. the functions $\psi_1$ and $\psi_2$ have a common real zero.

Further we use the notation $\xi_k^*(v)$
%and $\xi_k(q)$
to denote the values
of the zeros $\xi_k^*$
%or $\xi_k$
as functions of the parameter $v$. In
%or $q$ respectively. In
particular, we write $\xi_{2k_2}^*(\tilde{q}_k)$ for
$\xi_{2k_2}^*(v)|_{v=\tilde{q}_k}$.
%and $\xi_{2k_2}(v_{k_2}^*)$ for $\xi_{2k_2}(q)|_{q=v_{k_2}^*}$.
We also consider a sequence of indices $k_2$ and the corresponding 
double positive zeros of $\psi_1$ and $\psi_2$ obtained for
$v=\tilde{q}_{k_2}$. The sequences of these zeros tend to $e^{\pi /2}$
as $k_2\rightarrow \infty$, see part (5) of
Remarks~\ref{remsspectrum}. Indeed, it is true that 

$$\lim_{k\rightarrow\infty}|\xi_{2k}^*(\tilde{q}_{k})|=e^{\pi}~~~\, {\rm and}~~~\, 
\lim_{k\rightarrow\infty}\tilde{q}_k=1^-~,~~~\, {\rm so}~~~\, 
\lim_{k_2\rightarrow\infty}\chi_{2k_2}(\tilde{q}_{k_2})=e^{\pi /2}~~
    {\rm (see~(\ref{equchimuxi})).}$$
    Similarly $\lim_{k_2\rightarrow\infty}\mu_{2k_2}(\tilde{q}_{k_2})=e^{\pi /2}$.
    One has (see (\ref{equstring}))

    $$|\xi_{2k_2}^*(v^*_{k_2})|>(v_{k_2}^*)^{-2k_2+1}>(\tilde{q}_{k_2})^{-2k_2+1}~.$$
The right-hand side tends to $e^{\pi}$ as $k_2\rightarrow\infty$. This can be
deduced from (\ref{equstring}) and (\ref{equasympt}),
see~Remarks~\ref{remsspectrum}. Hence the sequence of numbers
$|\xi_{2k_2}^*(v^*_{k_2})|$ has an accumulation point which is $\geq e^{\pi}$
(we admit the possibility this point to be~$\infty$). 

Recall that $\chi_k=v^{1/8}(-\xi_k^*)^{1/2}$. We show below that one can choose
the sequence of indices $k_2$ such that $\lim_{k_2\rightarrow\infty}v^*_{k_2}=1^-$. 
Hence there exists an accumulation point of the sequence of
points $\chi_{2k_2}(v^*_{k_2})$
(i.~e. of the common zeros of $\psi_1$ and $\psi_2$)
which is $\geq e^{\pi /2}$.  

Now about the choice of the sequence of indices $k_2$. We denote it by
$k_2^1$, $k_2^2$, $\ldots$. When $k_2^1$ is chosen, then one first
increases $v$ until the value $\tilde{q}_{k_2^1}$ and then chooses the index
$k_2^2$ such that for $v=\tilde{q}_{k_2^1}$ and for $k=2k_2^2$, the inequalities
(\ref{equchi}) hold true. Thus $v_{k_2^2}^*>\tilde{q}_{k_2^1}>v_{k_2^1}^*$,
the sequence
$v_{k_2^s}^*$ is minorized by the sequence $\tilde{q}_{k_2^{s-1}}$ and the latter
tends to~$1^-$. This proves part~(1) of Theorem~\ref{tm2}, with
$q_s=(v_{k_2^s}^*)^{1/4}$.

\subsection{Proof of part (2) of Theorem~\ref{tm2}}

We need the following lemma:

\begin{lm}\label{lm02}
  With the notation of Remarks~\ref{remsspectrum} and 
  for $q\in (0,0.2]$, one has

$$\xi_{2k+1}\in (-q^{-2k-1.2},-q^{-2k-1})~~~\, 
{\rm and}~~~\, \xi_{2k+2}\in (-q^{-2k-2},-q^{-2k-1.8})~,~~~\, k=0,~1,~\ldots ~.$$
\end{lm}

\begin{proof}
  For $k=0$, we consider the functions $\tau_1(q):=\theta (q,-q^{-1.2})$ and
  $\tau_2(q):=\theta (q,-q^{-1.8})$:

  $$\begin{array}{ll}
    \tau_1=1-q^{-0.2}+A~,&A:=q^{0.6}-q^{2.4}+q^{5.2}-q^{9}+\cdots ~,\\ \\
\tau_2=1-q^{-0.8}+B~,&B:=q^{-0.6}-q^{0.6}+q^{2.8}-q^{6}+\cdots ~.\end{array}
    $$
For $q=0.2$, the sums of the Leibniz series $A$ and $B$ equal
$0.3599499830\ldots$ and $2.256770928\ldots$ respectively and one finds that 

$$\tau_1(0.2)=-0.0197796780\ldots <0~~~\, {\rm and}~~~\,
\tau_2(0.2)=-0.367127390\ldots <0~.$$ 
It is easy to show that both functions $\tau_1$ and $\tau_2$ are increasing for
$q\in (0,0.2]$, so one has 

  $$\xi_1\in (-q^{-1.2},-q^{-1})~~~\, {\rm and}~~~\, 
  \xi_2\in (-q^{-2},-q^{-1.8})~,$$
  see 2. and 3. of part (3) of Remarks~\ref{remsspectrum} and the inequalities
  (\ref{equstring}).
  Recall that (see~(\ref{equqqx}))

  $$\theta (q,x)=1+qx\theta(q,qx)=1+qx+q^3x^2\theta (q,q^2x)~.$$
  Then

  $$\begin{array}{ll}
    \theta (q,-q^{-2k-1.2})=1-q^{-2k-0.2}+q^{0.6-4k}\theta (q,-q^{-2(k-1)-1.2})&
           {\rm and}\\ \\
           \theta (q,-q^{-2k-1.8})=1-q^{-2k-0.8}+q^{-0.6-4k}\theta (q,-q^{-2(k-1)-1.8})~,&
           \end{array}$$
  so if for $q\in (0,0.2]$,

$$\theta (q,-q^{-2(k-1)-1.2})<0~~~\, {\rm and}~~~\, 
\theta (q,-q^{-2(k-1)-1.8})<0~,$$
then
$\theta (q,-q^{-2k-1.2})<0$ and $\theta (q,-q^{-2k-1.8})<0$. 
Thus by induction on $k$ and using (\ref{equstring}) the lemma is proved. 
\end{proof}

Next we use equations~(\ref{equevenodd}), (\ref{equv}) and~(\ref{equReIm})
and the corresponding notation.
The zeros of $\theta$ depend continuously on the parameter~$q$ hence on~$v$.
For $q\in (0,\tilde{q}_1]$, all zeros are real negative.
  If for some value $v_{\dagger}\in (0,0.2]$ of $v$,
  the function $\theta$ has complex 
conjugate zeros
with positive real part, then for some $v=v_1\in (0,v_{\dagger})$, it has a pair
of purely imaginary zeros and the functions $\psi_1(v_1,.)$ and
$\psi_2(v_1,.)$ have a common zero.

However for $v\in (0,0.2]$, such a common zero is impossible.
    Indeed, if one denotes as in the previous subsection 
by $\pm \chi_{j}$ and $\pm \mu_{j}$ the real zeros of
the functions $\psi_1(v,.)$ and $\psi_2(v,.)$, where 
$0<\chi_{j}<\chi_{j+1}$ and $0<\mu_{j}<\mu_{j+1}$, and by $-\xi_j^*$
the zeros of $\theta (v,.)$, then one has~(\ref{equchimuxi}).

%$$\chi_{j}=v^{1/8}\sqrt{-\xi_j^*}~~~\, {\rm and}~~~\,
%\mu_{j}=v^{-1/8}\sqrt{-\xi_j^*}~,~~~\,
%{\rm so}~~~\, \mu_{j}=v^{-1/4}\chi_{j}~.$$ 
One cannot have $\chi_{j}=\mu_{s}$ for $s\geq j$, because $v\in (0,1)$
and $\mu_{s}\geq \mu_{j}=v^{-1/4}\chi_{j}>\chi_{j}$.
This is also impossible for $s\leq j-3$,
because $\mu_s\leq \mu_{j-3}$ and by (\ref{equstring}) one has

$$\mu_{j-3}=v^{-1/8}\sqrt{-\xi_{j-3}^*}<v^{-1/8-(j-2)/2}=v^{-j/2+7/8}<
v^{-j/2+5/8}=v^{1/8-(j-1)/2}<v^{1/8}\sqrt{-\xi_{j}^*}=\chi_j~.$$
%~~~\, {\rm and}~~~\,
%\chi_{j}=v^{1/8}\sqrt{-\xi_{j}^*}>v^{1/8-(j-1)/2}=v^{-j/2+5/8}~.$$
For $s=j-2$, if $j$ is odd, then (see~(\ref{equstring}))

$$\mu_{j-2}=v^{-1/8}\sqrt{-\xi_{j-2}^*}<v^{-1/8-(j-1)/2}<v^{1/8-j/2}<
v^{1/8}\sqrt{-\xi_{j}^*}=\chi_{j}~.$$
If $j$ is even, then

$$\mu_{j-2}=v^{-1/8}\sqrt{-\xi_{j-2}^*}<v^{-1/8-(j-2)/2}<v^{1/8-(j-1)/2}<
v^{1/8}\sqrt{-\xi_{j}^*}=\chi_{j}~.$$
In the remaining case $s=j-1$ we use the condition $v\in (0,0.2]$ and
    Lemma~\ref{lm02}. If $j$ is odd, then

    $$\mu_{j-1}=v^{-1/8}\sqrt{-\xi_{j-1}^*}<v^{-1/8-(j-1)/2}<v^{1/8-j/2}<
    v^{1/8}\sqrt{-\xi_{j}^*}=\chi_{j}~.$$
    If $j$ is even, then

    $$\mu_{j-1}=v^{-1/8}\sqrt{-\xi_{j-1}^*}<v^{-1/8-(j-1+0.2)/2}<v^{1/8-(j-0.2)/2}<
    v^{1/8}\sqrt{-\xi_{j}^*}=\chi_{j}~.$$
    Thus for $v\in (0,0.2]$, i.e., for
      $q\in (0,0.2^{1/4}=0.6687403050\ldots ]$, there is no
      purely imaginary zero of
      $\theta$. The zeros of $\theta$ depending continuously on~$q$ and
      being all real negative for
      $q\in(0,\tilde{q}_1)$, for
      $q\in (0,0.2^{1/4}]$, they are either real negative or
        complex with negative real part.

        \begin{rem}\label{remnum}
          {\rm For $q=0.726$, $0.727$ and $\tilde{q}_1^{1/4}=0.7457222066\ldots$,
            the truncation $\theta_{100}$ of $\theta$
            (see part (6) of Remarks~\ref{remsKats}) has the following
            complex conjugate pairs of zeros respectively (they belong to the
            interior of Katsnelson's contour):}

          $$\begin{array}{cc}-0.004522146605\ldots \pm 2.911439535\ldots i~,&
            0.005050176876\ldots \pm 2.904960208\ldots i\\ \\
            {\rm and}&0.1780767569\ldots \pm 2.779382065\ldots i~.
          \end{array}$$
          {\rm These are the zeros closest to the imaginary axis. This makes
            one suppose that Theorem~\ref{tm2} could hold true with
            $0.726$ instead of $0.2^{1/4}$ in its formulation.}
          \end{rem}

        \section{Proof of Theorem~\protect\ref{tm3}\protect\label{secprtm3}}

        \subsection{Beginning of the proof of Theorem~\protect\ref{tm3}}

We remind that $y_s$ denotes the double zero of
$\theta (\tilde{q}_s,.)$, see Section~\ref{seczerosspectrum}. 
By \cite[Theorem~5]{KoSe}, for $s\geq 15$, one has $y_s>-38.9$. Any real
zero of $\theta$ is $<-5$ (see \cite[Proposition~2]{KoMatStud}). Hence
for $s\geq 15$, all
double zeros of $\theta$ belong to the interval $(-38.9,-5)$.

  Suppose that $q\in [\tilde{q}_s,\tilde{q}_{s+1})$. Then:
    \vspace{1mm}
    
    1) for
  $x\in [-q^{-2s-1},-q^{-2s+1}]$, one has $\theta (q,x)\geq 0$, with equality
    only for $(q,x)=(\tilde{q}_s,y_s)$; 
    \vspace{1mm}

     2) there exists $x_{\dagger}\in (-q^{-2s-2},-q^{-2s-1})$ such that
    $\theta (q,x_{\dagger})<0$;
    \vspace{1mm}
    
    3) this means that $x_{\dagger}/q\in (-q^{-2s-3},-q^{-2s-2})$ and
    $\theta (q,x_{\dagger}/q)=1+qx_{\dagger}\theta (q,x_{\dagger})>1$,
    see~(\ref{equqqx}).

\begin{lm}\label{lmx0}
  Suppose that for $q\in (0,1)$ fixed and for $x=x_0<-5$, one has
  $|\theta (q,x_0)|\geq 1$. Then one has $\theta (q,z)\neq 0$ for
  $|z|=|x_0|$.
\end{lm}

The lemma is proved in Subsection~\ref{subsecprlmx0}. It
implies that for $q\in [\tilde{q}_s,\tilde{q}_{s+1})$, no complex
  zero of
  $\theta$ crosses the circumference $\mathcal{C}(0,q^{-2s-3})$
  centered at $0$ and of radius $q^{-2s-3}$. As $q$ grows from $\tilde{q}_s$ to
  $\tilde{q}_{s+1}$, the quantity $q^{-2s-3}$ decreases.
  Therefore to find a disk
  centered at $0$ and containing all complex zeros of $\theta$ it suffices
  to find a majoration of the quantities $\tilde{q}_s^{-2s-3}$.

  From the fourth to the 25th spectral value (see part (2) of
  Remarks~\ref{remsspectrum}), the corresponding quantities
  $\tilde{q}_s^{-2s-3}$ are $<49.6$ (this can be checked directly).
  For $s\geq 26$, one knows that

  $$|y_s|<38.9~,~~~\, y_s\in (-q^{-2s},-q^{-2s+1})~~~\, {\rm and}~~~\, 
  \tilde{q}_s\in (\tilde{q}_{25},1)~,~~~\, {\rm with}~~~\,
  \tilde{q}_{25}=0.940393\ldots ~.$$
  That's why for $s\geq 26$,

  $$\tilde{q}_s^{-2s-3}=\tilde{q}_s^{-2s+1}\cdot \tilde{q}_s^{-4}<
  |y_s|\cdot \tilde{q}_s^{-4}<
  38.9\cdot (\tilde{q}_{25})^{-4}<
  38.9\cdot 0.940393^{-4}=49.74\ldots <49.8~.$$
  Thus the inequality $\tilde{q}_s^{-2s-3}<49.8$ is proved for $s\geq 4$. The
  quantities $q^{-2s-3}$ being decreasing, we have proved the following lemma:

  \begin{lm}\label{lms}
    For $s\geq 4$ and $q\in [\tilde{q}_s,\tilde{q}_{s+1})$, no complex zero
      of $\theta$ belongs to $\mathcal{C}(0,q^{-2s-3})$.
  \end{lm}
  
  In particular, this
  means that the complex conjugate pairs born from the double zeros $y_4$,
  $y_5$, $y_6$, $\ldots$ lie inside the circumference $\mathcal{C}(0,49.8)$.

\subsection{Completion of the proof of Theorem~\protect\ref{tm3}}

For the remaining three conjugate pairs of zeros (born from the double
zeros $y_1$, $y_2$ and $y_3$) we modify our method by
  considering not only integer, but also half-integer negative powers of~$q$.

  \begin{nota}
    {\rm We denote by $c_j$ the complex conjugate pair born from the
      double zero $y_j$ for $q=\tilde{q}_j$, and by $c_j^+$ (resp. $c_j^-$)
      the zero of this pair with positive (resp. negative) imaginary part.}
  \end{nota}

  {\em The pair $c_3$.} One finds numerically that $q^{-15/2}<49.8$ if and
  only if $q>0.59\ldots$ and that $\theta (q,-q^{-15/2})<-1$ if
  $q<0.66393\ldots$.
  The pair $c_3$ is born for $q=\tilde{q}_3=0.630628\ldots$. Hence for
  $q\in (\tilde{q}_3,0.663]=:\mathcal{I}$,
    the pair $c_3$ is inside the circumference
  $\mathcal{C}(0,q^{-15/2})$ which in turn is inside the circumference
  $\mathcal{C}(0,49.8)$.

  One has $q^{-19/2}<49.8$ if and only if $q>0.66274\ldots$; it is true that
  $\theta (q,-q^{-19/2})<-1$ if $q<0.72344\ldots$. Thus for
  $q\in [0.6628,0.72]=:\mathcal{J}$,
  the pair $c_3$ is inside 
  $\mathcal{C}(0,q^{-19/2})$ which is inside 
  $\mathcal{C}(0,49.8)$.

  One has $q^{-21/2}<49.8$ if and only if $q>0.6892\ldots$, and
  $\theta (q,-q^{-21/2})>1$ if $q<0.7465\ldots$. Therefore 
  for $q\in [0.6893,0.7465]=:\mathcal{K}$, the pair $c_3$ is inside
  $\mathcal{C}(0,q^{-21/2})$ which is inside 
  $\mathcal{C}(0,49.8)$. One notices that
  $\mathcal{I}\cap \mathcal{J}\neq \emptyset \neq \mathcal{J}\cap \mathcal{K}$. 

  But the circumference $\mathcal{C}(0,q^{-21/2})$ is inside
  $\mathcal{C}(0,q^{-11})$ which is inside $\mathcal{C}(0,49.8)$ for
  $q\in [\tilde{q}_4,\tilde{q}_5)$ and $\tilde{q}_4=0.7012\ldots <0.7465$,
    see Lemma~\ref{lms} with $s=4$; the lemma implies that the pair $c_3$ is
    inside $\mathcal{C}(0,q^{-2s-3})$ for $q\in [\tilde{q}_s,\tilde{q}_{s+1})$,
      $s\geq 4$. Thus for $q\in [\tilde{q}_4,1)$ (hence
      for $q\in (\tilde{q}_3,1)$), the pair $c_3$ is inside
      $\mathcal{C}(0,49.8)$.
      \vspace{1mm}

{\em The pair $c_2$.} This pair is born for
      $q=\tilde{q}_2=0.516959\ldots$. The conditions $q^{-11/2}<49.8$ and
      $\theta (q,-q^{-11/2})<-1$ are fulfilled for $q>0.4913\ldots$ and
      $q<0.5721\ldots$ respectively. Hence for
      $q\in (\tilde{q}_2,0.5721]\subset [0.4914,0.5721]$,
      the pair $c_2$ is inside 
  $\mathcal{C}(0,q^{-11/2})$ which is inside 
      $\mathcal{C}(0,49.8)$.

      Next, the inequalities $q^{-13/2}<49.8$ and
      $\theta (q,-q^{-13/2})>1$ hold true for $q>0.5481\ldots$ and
      $q<0.6249\ldots$ respectively, so for $q\in [0.5482,0.6249]$, the pair
      $c_2$ is inside 
  $\mathcal{C}(0,q^{-13/2})$ which is inside 
      $\mathcal{C}(0,49.8)$.

In the same way as for the pair $c_3$ one shows that for
      $q\in [0.6,0.663]$, the pair $c_2$ is inside $\mathcal{C}(0,q^{-15/2})$
      and for $q\in [0.6628,0.72]$, it is inside $\mathcal{C}(0,q^{-19/2})$.
      In both cases these circumferences are inside $\mathcal{C}(0,49.8)$.
      And then one repeats the last two paragraphs of the proof
      that the pair $c_3$
      is inside $\mathcal{C}(0,49.8)$ for $q\in (\tilde{q}_3,1)$.
      \vspace{1mm}

      {\em The pair $c_1$.} We set $\rho_0:=2^{11/2}=45.2548\ldots$.
      We show first that
      for $q\in (\tilde{q}_1,1/2]$, one has $|c_1^{\pm}|<\rho _0$. Suppose that
      for some $q\in (\tilde{q}_1,1/2]$, one has $|c_1^{\pm}|\geq \rho _0$.
    Recall that (see~\cite{KoDBAN14})

    $$\theta (q,x)=\prod_{j=1}^{\infty}(1-x/\xi_j)=\sum_{j=0}^{\infty}q^{j(j+1)/2}x^j~,$$
    where $\xi_j$ are
    the zeros of $\theta (q,.)$ counted with multiplicity. Then the coefficient
    of $x^2$ equals
    
 $$r_2:=q^3=\sum_{1\leq k<m}1/\xi_k\xi_m=:s_2~.$$
For $q\in (\tilde{q}_1,1/2]$, one obtains
  $\tilde{q}_1^3=0.02957\ldots \leq r_2\leq 1/8=0.125$. We remind that
for the real zeros of $\theta$ the following inequalities hold true
(see~(\ref{equstring})):

  $$-q^{2j}<\xi_{2j}<\xi_{2j-1}<-q^{2j-1}~.$$
  In the sum $s_2$ we set $\xi_1:=c_1^+$ and $\xi_2:=c_1^-$, so

  \begin{equation}\label{equs2}
    s_2=1/|c_1^+|^2+((1/c_1^++1/c_1^-)\sum_{k=3}^{\infty}1/\xi_k)+
  \sum_{3\leq k<m}1/\xi_k\xi_m~.\end{equation}
  One obtains a majoration of the rightmost sum by setting
  $\xi_{2j}=\xi_{2j-1}=-q^{2j-1}$, $j=2$, $3$, $\ldots$. In this case the sum
  equals

  $$\begin{array}{ccl}
    \phi (q)&:=&\sum_{j=1}^{\infty}q^{2j+1}(q^{2j+1}+2q^{2j+3}+2q^{2j+5}+\cdots )\\ \\
    &=&-\sum_{j=1}^{\infty}(q^{2j+1})^2+
    2\sum_{j=1}^{\infty}q^{2j+1}(q^{2j+1}+q^{2j+3}+q^{2j+5}+\cdots)\\ \\
    &=&-q^6/(1-q^4)+2\sum_{j=1}^{\infty}q^{4j+2}/(1-q^2)=
    q^6(1+q^2)/((1-q^2)(1-q^4))~.
  \end{array}$$
  The function $\phi$ is increasing on the interval $[0.3,0.5]$, with
  $\phi (0.3)=0.0008803\ldots$ and $\phi (0.5)=0.0277\ldots$.

  With regard to (\ref{equs2}), for $|c_1^{\pm}|\geq \rho_0$,
  one obtains $1/|c_1^+|^2\leq 1/\rho_0^2=0.000488\ldots$.

  The
  product $(1/c_1^++1/c_1^-)\sum_{k=3}^{\infty}1/\xi_k$ is maximal for
  $c_1^+=c_1^-=-\rho_0$ (so $1/c_1^++1/c_1^-=2^{-9/2}$)
  and $\xi_{2j}=\xi_{2j-1}=-q^{2j-1}$, $j=2$, $3$, $\ldots$,
  in which case it equals

  $$\delta (q):=2^{-9/2}(2q^3/(1-q^2))~.$$
  The difference $q^3-\phi (q)-1/\rho_0^2-\delta (q)$ is positive on
  $[0.3,0.5]$; its minimal value $0.0230\ldots$ is attained for $q=0.3$. 
  Hence the equality $r_2=s_2$ is impossible for $|c_1^{\pm}|\geq \rho_0$.

Thus for $q\in (\tilde{q}_1,1/2]$, the pair $c_1$ is inside the circumference
  $\mathcal{C}(0,q^{-11/2})$. In the same way as for the pair $c_2$ one shows
  that for $q\in [0.4914, 0.5721]$, the pair $c_1$ is inside
  $\mathcal{C}(0,q^{-11/2})$ and
  for $q\in [0.5482,0.6249]$, it is inside $\mathcal{C}(0,q^{-13/2})$;
  as this was shown for the pair $c_3$ one proves that the pair $c_1$ is inside
  $\mathcal{C}(0,q^{-15/2})$ for $q\in [0.6,0.663]$ and inside
  $\mathcal{C}(0,q^{-19/2})$ for $q\in [0.6628,0.72]$; hence in all these cases
  the pair $c_1$ is 
  inside $\mathcal{C}(0,49.8)$. 
  Then by Lemma~\ref{lms} this pair is inside
  $\mathcal{C}(0,49.8)$ for $q\in (\tilde{q}_1,1)$.

\subsection{Proof of Lemma~\protect\ref{lmx0}\protect\label{subsecprlmx0}}

  We remind that for $m\in \mathbb{N}^*$, one has
  $\theta (q,-q^{-m})\in (0,q^m)$ (see \cite[Proposition~9]{KoBSM1}),
  so $|\theta (q,-q^{-m})|<1$. Hence if
  $|\theta (q,x)|\geq 1$, then $x\neq -q^{-m}$. We use the Jacobi triple
  product, see the definition of the function $\Theta^*$
  in Subsection~\ref{subsecmethod}. Consider the
  product

  $$T_m(q,z):=(1+q^mz)(1+q^m/z)=q^m(q^{-m}+z)(1/z)(z+q^m)$$
  for $|z|=|x_0|$. The modulus $|q^{-m}+z|$ is minimal for $z=x_0$. Indeed,
  we saw that $x_0\neq -q^{-m}$. Consider
  the circumference $\mathcal{C}_1:=\mathcal{C}(0,|x_0|)$. 
  
  If the point $-q^{-m}$ is inside $\mathcal{C}_1$ (i.~e. $|x_0|>q^{-m}$),
  then the circumferences
  $\mathcal{C}_1$ and $\mathcal{C}_2:=\mathcal{C}(-q^{-m},|x_0|-q^{-m})$
  are tangent at $x_0$ and $\mathcal{C}_2$ is inside $\mathcal{C}_1$. Hence
  for $|z|=|x_0|$ and $z\neq x_0$, the point $z$ is outside $\mathcal{C}_2$, so

  $$|z+q^{-m}|>|x_0+q^{-m}|~.$$
  The circumference
  $\mathcal{C}_3:=\mathcal{C}(-q^{m},|x_0|-q^{m})$ is also inside $\mathcal{C}_1$
  and tangent to it at $x_0$, therefore

  $$|z+q^{m}|>|x_0+q^{m}|~~~\, {\rm and}~~~\, 
  |T_m(q,z)|>|T_m(q,x_0)|~.$$

  If the point $-q^{-m}$ is outside $\mathcal{C}_1$ (i.~e. $|x_0|<q^{-m}$),
  then the inequality
  $|z+q^{m}|>|x_0+q^{m}|$ is proved in the same way. To show that
  $|z+q^{-m}|>|x_0+q^{-m}|$, we consider the tangent line $\mathcal{L}$
  to $\mathcal{C}_1$
  at $x_0$ and the point $z'\in \mathcal{L}$ such that Im$\, z'=$Im$\, z$. Then

  $$|z+q^{-m}|>|z'+q^{-m}|>|x_0+q^{-m}|$$
  and again $|T_m(q,z)|>|T_m(q,x_0)|$. This means that
  $|\Theta^*(q,z)|>|\Theta^*(q,x_0)|$.

  Recall that $\theta =\Theta^*-G$. For $x<-5$, the series of $G$
  is a Leibniz series with a negative first term, so $0<-G<1/5$. This is why the
  inequality $\theta (q,x_0)\leq -1$ implies $\Theta^*(q,x_0)<-1$ hence
  $|\Theta^*(q,z)|>1$ for $|z|=|x_0|$. As $|G(q,z)|<\sum_{j=1}^{\infty}1/5^j=1/4$,
  one obtains
  $|\theta (q,z)|>3/4$ and $\theta (q,.)$ has no zero on~$\mathcal{C}_1$.

  If $\theta (q,x_0)\geq 1$, then $\Theta^*(q,x_0)>4/5$, so
  $|\Theta^*(q,z)|>4/5$ and $|\theta (q,z)|>4/5-1/4=11/20$ for $|z|=|x_0|$.


\begin{thebibliography}{40}
\bibitem{AnBe} G. E. Andrews, B. C. Berndt,  
Ramanujan's lost notebook. Part II. Springer, NY, 2009.
\bibitem{BeKi} B. C. Berndt, B. Kim, 
Asymptotic expansions of certain partial theta functions. 
Proc. Amer. Math. Soc. 139:11 (2011), 3779-3788.
\bibitem{BFM} K.~Bringmann, A.~Folsom and A.~Milas, 
Asymptotic behavior 
of partial and false theta functions arising from Jacobi forms and regularized 
characters. 
J. Math. Phys. 58:1 (2017), 011702, 19 pp.
\bibitem{BrFoRh} K. Bringmann, A. Folsom, R. C. Rhoades, 
Partial theta functions 
and mock modular forms as $q$-hypergeometric series, 
Ramanujan J. 29:1-3 (2012), 
295-310, 
http://arxiv.org/abs/1109.6560
\bibitem{CMW} T.~Creutzig, A.~Milas and S.~Wood, 
On regularised quantum 
dimensions of the singlet vertex operator algebra and false theta functions. 
Int. Math. Res. Not. 5 (2017), 1390-1432.
\bibitem{FGM} R. Flores and J. Gonz\'alez-Meneses,
  On the growth of Artin-Tits monoids and the partial theta function. 
J. Combin. Theory Ser. A 190 (2022), Paper No. 105623, 39 pp.
\bibitem{Ha} G.~H.~Hardy, 
On the zeros of a class of integral functions, 
Messenger of Mathematics, 34 (1904), 97-101.
\bibitem{Hu} J. I. Hutchinson, 
On a remarkable class of entire functions, 
Trans. Amer. Math. Soc. 25 (1923), 325-332.
\bibitem{KaLoVi} O.M. Katkova, T. Lobova and A.M. Vishnyakova, 
On power series 
having sections with only real zeros. 
Comput. Methods Funct. Theory 3:2 (2003), 
425-441.
\bibitem{Ka} B. Katsnelson, On summation of the Taylor series 
of the function 
$1/(1-z)$ by the theta summation method,
Complex analysis and dynamical systems VI. Part 2, 141–157.
Contemp. Math., 667
Israel Math. Conf. Proc.
American Mathematical Society, Providence, RI, 2016
arXiv:1402.4629v1 [math.CA]. 
%\bibitem{Ko1} V.P. Kostov, About a partial theta function, 
%Comptes Rendus Acad. Sci. Bulgare 66, No 5 (2013) 629-634.
\bibitem{KoFAA} V.P. Kostov, 
On the multiple zeros of a partial theta function, 
Funct. Anal. Appl. 
(Russian version 50, No. 2 (2016) 84-88, 
English version 50, No. 2 (2016) 153-156).
\bibitem{KoDBAN14} V.P. Kostov, A property of a partial theta function,
  Comptes Rendus Acad. Sci. 
Bulgare 67, No. 10 (2014) 1319-1326.
\bibitem{KoFAA19} V.P. Kostov, On the complex conjugate zeros of the partial
  theta function, Funct. 
  Anal. Appl. (English version 2019, 53:2, 149-152,
  Russian version 2019, 53:2, 87-91).
\bibitem{KoBSM1} V.P. Kostov, 
On the zeros of a partial theta function, 
Bull. 
Sci. Math. 137:8 (2013), 1018-1030.
\bibitem{KoBSM2}  V.P. Kostov, On the double zeros of a partial theta function, 
Bull. Sci. Math. 140, No. 4 (2016) 98-111.
%\bibitem{Ko3} V.P. Kostov, Asymptotics of the spectrum of partial theta 
%function, Revista Mat. Complut. 27, No. 2 (2014) 677-684, 
%DOI: 10.1007/s13163-013-0133-3.
%\bibitem{KoPRSE1} V.P. Kostov, On the spectrum of a partial theta function, 
%Proc. Royal Soc. Edinb. A 144 (2014) No. 5, 925-933.
\bibitem{KoPRSE2}V.P. Kostov,  
On a partial theta function and its spectrum, 
Proc. Royal Soc. Edinb. A 146:3 (2016), 609-623.
%\bibitem{Ko5} V.P. Kostov, A property of a partial theta function, 
%Comptes Rendus Acad. Sci. Bulgare 67, No 10 (2014) 1319-1326.
\bibitem{KoPMD} V.P. Kostov, 
A domain containing all zeros of the partial 
theta function, 
Publ. Math. Debrecen 93:1-2 (2018), 189-203.
\bibitem{KoAA} V.P. Kostov, 
A separation in modulus property of the 
zeros of a partial theta function, 
Analysis Mathematica 44:4 (2018), 
501-519.
\bibitem{KoSe}V.P. Kostov, On the zero set of the partial theta function,
  Serdica Math. J. 45 (2019), 225-258.
\bibitem{KoMatStud} V.P. Kostov, A domain free of the zeros of the
  partial theta function.
  Mat. Stud. 58 (2022), no. 2, 142-158.
\bibitem{KoAnn24} V.P. Kostov, No zeros of the partial theta function in the
  unit disk, Annual of Sofia University ``St. Kliment Ohridski'', Faculty
  of Mathematics and Informatics 111 (2024) 129-137.
  DOI: 10.60063/gsu.fmi.111.129-137
\bibitem{KoSh} V.P. Kostov and B. Shapiro, 
Hardy-Petrovitch-Hutchinson's 
problem and partial theta function, 
Duke Math. J. 162:5 (2013), 825-861. 
arXiv:1106.6262v1[math.CA].
%\bibitem{Le} B. Ja. Levin, Zeros of Entire Functions, AMS, Providence, 
%Rhode Island 1964.
\bibitem{LuSa} D.S. Lubinsky, E. Saff, Convergence of Pad\'e approximants
  of partial theta functions and the Rogers-Szeg\H{o}
polynomials, Constructive Approximation, 3 (1987), 331-361.
\bibitem{EM} E.T. Mortenson, On the dual nature of partial theta functions
  and Appell-Lerch sums, Adv. Math. 264 (2014), 236-260.
\bibitem{Ost} I.~V.~Ostrovskii, 
On zero distribution of sections and tails 
of power series, 
Israel Math. Conf. Proceedings, 15 (2001), 297-310.
\bibitem{Pe} M.~Petrovitch, 
Une classe remarquable de s\'eries enti\`eres, 
Atti del IV Congresso Internationale dei Matematici, Rome (Ser. 1), 2 
(1908), 36-43. 
%\bibitem{PolSch14}
%G. P\'olya and J. Schur, \"Uber zwei Arten von Faktorenfolgen in der
%Theorie der algebraischen Gleichungen, J. Reine Angew. Math. 144
%(1914) 89--113.
\bibitem{Pr} T.~Prellberg, The combinatorics of the leading root of the 
partial theta function, http://arxiv.org/pdf/1210.0095.pdf
%\bibitem{PoSz} G. P\'{o}lya, G. Szeg\H{o}, 
%Problems and Theorems in Analysis, Vol. 1, Springer, Heidelberg 1976.
\bibitem{So} A.~Sokal, 
The leading root of the partial theta function, 
Adv. Math. 229:5 (2012), 2603-2621. 
arXiv:1106.1003.
\bibitem{Sun} L. H. Sun, An extension of the Andrews-Warnaar partial theta
  function identity. 
Adv. in Appl. Math. 115 (2020), 101985, 20 pp.
\bibitem{WM} J. Wang and X. Ma, On the Andrews-Warnaar identities
  for partial theta functions. Adv. in Appl. Math. 97 (2018), 36-53.
\bibitem{Wa} S. O. Warnaar, 
Partial theta functions. I. Beyond 
the lost notebook, 
Proc. London Math. Soc. (3) 87:2 (2003), 363-395.
\bibitem{Wei} C. Wei, Partial theta function identities from Wang and Ma's
  conjecture.
J. Difference Equ. Appl. 26 (2020), no. 4, 532-539.
%\bibitem{Wi} Wikipedia. Jacobi theta function.
\end{thebibliography}
\end{document}